\documentclass[11pt,a4paper]{article}
\usepackage[margin=2.5 cm]{geometry}

\usepackage[T1]{fontenc}
\usepackage{lmodern}
\usepackage{bbm} 
\usepackage{tikz}
\usetikzlibrary{automata,arrows,positioning,calc}
\usepackage{amsmath} 
\usepackage{empheq}
\usepackage{amssymb}
\usepackage{amsfonts}
\usepackage{dsfont}
\usepackage[english]{babel}
\usepackage[T1]{fontenc}
\usepackage{dsfont}
\usepackage{dcolumn}
\usepackage{physics}
\usepackage{subcaption}
\usepackage{mwe}
\usepackage{array, booktabs, makecell}

\usepackage{siunitx, mhchem}
\usepackage{todonotes}
\usepackage{tikz-cd}
\usepackage{enumitem}
\usepackage{bm}
\usepackage{amsmath}

\usepackage{hyperref}

\hypersetup{%
    pdfpagemode={UseOutlines},
    bookmarksopen,
    pdfstartview={FitH},
    colorlinks,
    linkcolor={black},
    citecolor={black},
    urlcolor={black}
  }

\usepackage{mathrsfs,amssymb,amsfonts,amsthm,amsmath,amsbsy}
\usepackage{natbib}
\newtheorem{theorem}{Theorem}[section]

\newtheorem{lemma}[theorem]{Lemma} 
 
\newtheorem{example}[theorem]{Example} 
\newtheorem{definition1}[theorem]{Definition} 
\newtheorem{remark1}[theorem]{Remark} 
\usepackage[ruled,vlined]{algorithm2e}
\usepackage[font=small]{caption}
\usepackage{breqn}
\usepackage[table,dvipsnames]{}
\usepackage{tikz-3dplot} 
\usepackage{amssymb}
\usepackage{xifthen}
\usepackage{multirow}
\usepackage{amscd}
\usepackage{mathtools}

\tdplotsetmaincoords{60}{125}
\tdplotsetrotatedcoords{0}{0}{0} 

\DeclareMathSizes{10}{10}{7}{6}

\newenvironment{nalign}{
    \begin{equation}
    \begin{aligned}
}{
    \end{aligned}
    \end{equation}
    \ignorespacesafterend
}

\newlength{\wordlength}

\begin{document}
\title{
Neural Control Systems}

\bigskip
\author{\textbf{\textbf{Paolo Colusso$\vphantom{l}^{1}$ and Damir Filipovic$\vphantom{l}^{1}$}} 
\\\\
$\vphantom{l}^{\text{1}}$Ecole Polytechnique F\'ed\'erale de Lausanne\\ and Swiss Finance Institute 
}

\date{16 April 2024}

\maketitle

\begin{abstract}
We propose a function learning method with a control-theoretical foundation. We para\-metrise the approximating function as the solution to a control system on a reproducing-kernel Hilbert space, and propose several algorithms to find the set of controls which bring the initial function as close as possible to the target function. At first, we derive the expression for the gradient of the cost function with respect to the controls that parametrise the control system of difference equations. This allows us to find the optimal controls by means of gradient descent. In addition, we show how to compute derivatives of the approximating functions with respect to the controls and describe two optimisation algorithms relying on linear approximations of the approximating functions. We show how the assumptions we make lead to results which are coherent with Pontryagin's maximum principle. We test the optimisation algorithms on two toy examples and on two higher-dimensional real-world problems, showing that our method succeeds in learning from real data and is versatile enough to tackle learning tasks of different nature.\\
\end{abstract}

\paragraph*{Key words:} machine learning, financial data, neural networks, control theory, optimisation, bilinear control systems. 
\paragraph*{AMS:} 91G20, 68T07, 68T09, 91G60, 93C10.

\section{Introduction}

We contribute to the research in learning algorithms, while aiming at bridging the gap between machine learning and control theory. We work in the realm of function learning: this can be defined as the task of modelling dependence relations between data which are – possibly and likely – multivariate and nonlinear, and it usually occurs by minimising a cost function which measures the distance between a target function and an approximating function. In the method that we propose, the approximating function is parametrised as the solution to a control system on a Hilbert space. We define the control system by means of difference equations in discrete time with a given initial condition. Our goal is to find a set of controls which bring the initial function as close as possible to the target function, and suitable optimisation algorithms to look for such optimal controls are proposed. Our method hence provides contributions in the following ways: it allows to exploit the expertise which has been developed in the last decades in control theory, which serves as its foundation; it extends theoretical results in optimal control to Hilbert spaces; and it channels the acquired knowledge into function-learning tasks, which our data-rich world needs to tackle.

The control system in our model is in discrete time and deterministic. This is one of the many options among those which have been studied in optimal control theory. Dynamic programming and the Bellman optimality principle have been used to determine a set of optimal controls which minimise a total cost along a trajectory in discrete time, where each control determines the next state and entails a cost. 
Deterministic optimal control leads to the maximum principle, which was introduced by Pontryagin (\cite{neustadt1962mathematical}), and which holds both in continuous time and in discrete time. The maximum principle is an ODE (or a difference equation in time in the case of discrete time systems) which characterises optimal trajectories. Interestingly, in our attempt of drawing a link between control theory and machine learning, we show how the assumptions we make lead to results which are coherent with Pontryagin's maximum principle. 

While coherent with classical results in control theory, our approach makes a step forward in integrating a machine-learning approach with a control-theoretical foundation. In the spirit of function learning, we propose a forward propagation defined by a system of difference equations in a reproducing kernel Hilbert space (RKHS). We derive the expression for the gradient of the cost function with respect to the controls that parametrise the difference equations and propose a gradient-descent based approach to find the optimal controls. In addition, we show how to compute derivatives of the approximating functions with respect to said controls and introduce two alternative optimisation algorithms relying on the linearisation of the approximating functions around a fixed control. We test the proposed optimisation algorithms on two toy examples and on two higher-dimensional real-world problems, showing that our method succeeds in learning from real data and is versatile enough to tackle learning tasks of different nature, both a regression and a classification problem.

The study of control systems has its foundations in optimisation theory and dynamical systems (\cite{bot22} and references therein), with applications ranging from cellular biology and epidemic management to the development of stable financial systems and efficient power grids (\cite{batti13}, \cite{Schneider_2022}, \cite{GOTTGENS20152614}, \cite{schaf2018}). Controllability, the property of driving an initial state to a desired target via a control system, however, suffers from several drawbacks, and practical constraints, energy and resource constraints to mention two, hinder its applications (\cite{PhysRevLett108}). Optimal control theory (see \cite{todorov06optimal}) focuses on the minimisation of a cost function rather than on controllability and relies on the pioneering works of Pontryagin (\cite{neustadt1962mathematical}) and Bellman (\cite{bellman1966dynamic}). In our approach, the control system mainly serves the purpose of parametrising the approximating function, instead of that of describing a physical system.

While (stochastic) optimal control problems are theoretically well known (\cite{Fleming75}, \cite{Bertsekas2017}), exact numerical methods for their solutions are yet to be extensively studied (\cite{bensoussan23}) and the literature mainly focuses on approximate solutions (\cite{oster2021approximating}, \cite{debrabant2014semilagrangian}, \cite{Kalise_2018}, \cite{horowitz2014linear}). Numerical methods can be shown to converge only under the discretisation of the state space, which is problematic in high-dimensional state spaces due to the curse of dimensionality. One approach which seems to be prevailing at the moment is the combination of control theory with machine learning, which leads to what is referred to as \emph{reinforcement learning} (see as a reference \cite{mitReinforcementLearning}, \cite{bertsekas2019reinforcement}, \cite{recht2018tour}). \cite{kim20} establish the link between Q-learning -- a field of reinforcement learning based on stochastic approximation -- and HJB equations. Our approach also bridges the gap between machine learning and control theory, as well as being robust when it comes to scaling to higher dimensions.

Historically, the study of control systems has mainly been focusing on linear control systems, although advancements in nonlinear analysis paved the way to nonlinear control system theory (\cite{sontag98}). Within the latter, the class of bilinear control systems has allowed for a rich yet simple theory, with several applications (\cite{elliott09} and the literature within). It has to be stressed, however, that the focus of the literature has mostly been on state spaces in $\mathds{R}^n$. In addition, the link between machine learning theory and control theory has yet to be fully explored, as it mostly relies on numerical approximations to date. There are exceptions. \cite{NEURIPS2020} propose to integrate optimal control theory into an end-to-end learning process, leading to an optimal-control-informed learning framework. Their work is however limited to $\mathds{R}^n$ and to linear systems. Furthermore, theory on control systems on Hilbert spaces exists, but it is largely related to the stability properties of the system (see for example \cite{feintuch2012robust}) and to its controllability (\cite{slemrod1974note}), and, again, mostly in the linear case. Not only is the focus of our approach on Hilbert spaces and on function learning, but the system that we consider is also a nonlinear one, making our approach a considerable extension of the existing literature in the field.

Our paper contributes to the literature by linking machine learning and control theory. On the one hand, the control system is used to parametrise the approximating function as part of the learning task; on the other hand, we show that our approach accommodates classical results in control theory, extending them to Hilbert spaces. Unlike reinforcement learning, which shares the goal of bridging the gap between machine learning and control theory, the control system with which we work is deterministic. In addition, it is in discrete time, and is shown to successfully scale to higher dimensions. Lastly, the algorithms that we propose are versatile enough to be adapted to a wide range of learning tasks.

\subsection{Motivation}\label{moti}

We consider an artificial neural network. This corresponds to the mapping of an input state $x \in \mathds{R}^d$ to an output state $y \in \mathds{R}$ defined by:
\[
x \mapsto y = \phi_0 \, \circ \, \phi_1 \, \circ \dots \, \circ \phi_T(x),
\]
where $T$ is the number of hidden layers $\phi_t(x) = \sigma(A_tx + b_t):  \mathds{R}^d \to \mathds{R}^d$, for an activation function $\sigma: \mathds{R}^d \to \mathds{R}^d$, weight matrices $A_t$ and bias terms $b_t$, $t = 1, \dots, T$. Lastly, $\phi_0: \mathds{R}^d \to \mathds{R}$ is a readout map.

Having set $z_0(x) = x$, the forward propagation through this network corresponds to the sequence of mappings $z_t:\mathds{R}^d \to \mathds{R}^d$:
\begin{align*} 
z_{t+1} &=  \phi_{T-t} \, \circ \dots \, \circ \phi_T \circ z_0,\\
t &= 0, \dots, T-1.
\end{align*}
In contrast, the backward propagation $h_t:\mathds{R}^d \to \mathds{R}^d$ is given by
\[
h_{t+1} =  \phi_{0} \, \circ \dots \, \circ \phi_{t+1},
\]
which leads to the linear dynamic system defined by
\begin{nalign}\label{map}
    h_{t+1} &= h_t \, \circ \phi_{t+1},\\
    h_0 &= \phi_0.
\end{nalign}

Notice that, although both propagations lead to the same terminal function, only the latter is linear in the functions $h_t$. However, research has mainly been focusing on the forward propagation $z_t$ and its continuous-time limit, see for instance \cite{Chen18}. In their work, as $T$ becomes larger, the incremental contribution introduced by each layer then results in a neural differential equation (ODE) in $\mathds{R}^d$. This has the form
\begin{align*} 
\frac{d}{dt}z^{}_t &= f_t(z_t;\theta) = \phi_{T-t+1}(z)\\
 z_0 &= x,
\end{align*}
for $f_t(,\theta): \mathds{R}^d \to \mathds{R}^d$ non-linear mappings which depend on a parameter $\theta$. Gradient descent is then used to train the neural ODE. In the same framework of forward propagations, both \cite{E17} and \cite{Cucchiero20} propose an approach which relies on controls. In their work, the right-hand side of the ODE depends on controls $u_t$. Under the assumptions of the Chow-Rashevskii theorem (see \cite{Montgomery02}), \cite{Cucchiero20} show the universal interpolation property, namely that for any finite set of input-output pairs $(x_i, y_i)$ in $\mathds{R}^d \times \mathds{R}^d$ there exist controls $(u_t)$ such that $z_T (x_i) = y_i$ for all $i$. However, no indication is provided as to how the optimal controls are to be found. In addition, in all cases, generalisations properties cannot be inferred on the mapping $x \mapsto z_T(x)$.

\section{Framework}\label{meth}
In the following, unless explicitly stated, we consider $t = 0, \dots, T-1$ and $i = 1, \dots, q$.

We formalise the linear dynamic system \eqref{map} as the forward propagation
\begin{nalign}\label{fwdp}
h_{t+1} & = P_t h_t,\\
h_0 & = \phi_0.
\end{nalign}
Let $\mathcal{H}$ be a reproducing kernel Hilbert space (RKHS) with kernel $k$ of functions defined on a state space $\mathcal{X}$.
We write the forward propagation \eqref{fwdp} as a system of difference equations in $\mathcal{H}$ and in discrete time. We index time by $t \in \{0,1, \dots, T\}$, $T \in \mathds{N}$, and obtain
\begin{nalign}\label{diffh}
\Delta h_{t+1} &= A[u_t]h_t,\\
 h_0 &= \phi_0.
\end{nalign}
In the above, we denote as $\Delta h_{t+1} = h_{t+1}-h_t$, and write 
\[
A[u] := u_1 A_1 + \dots + u_q A_q,
\]
where $A_i$ are bounded linear operators on $\mathcal{H}$. The controls are indexed in discrete time, namely for $t \in \{0,1, \dots, T-1\}$, $u = (u_0, \dots, u_{T-1}) \in \mathds{R}^{q \times T}$. Hence, at time $t = T$ the solution of the difference equation is a mapping
\[
h_T(\cdot): \mathds{R}^{q \times T} \to \mathcal{H}.
\]
\begin{lemma}\label{lemmasolh}
The solution $h = h(u)$ for a given control $u = (u_t)$ of the difference equation \eqref{diffh} is 
\begin{equation}\label{solh}
\notag
h_t =  \prod_{s=0}^{t-1} \left(I +  \sum_{i=1}^q A_i u_{i,s}\right)h_0, \quad t = 1, \dots, T.
\end{equation}
\end{lemma}
\begin{proof}
It follows immediately from \eqref{diffh} that function $h_1$ is given by
\[
h_1 = h_0 + A[u_0] h_0 = h_0 + \sum_{i=1}^q A_i u_{i,0}h_{0},
\]
while function $h_2$ by
\begin{align*}
h_2 &= h_1 + \sum_{i=1}^q A_i u_{i,1}h_1 = \left(I +  \sum_{i=1}^q A_i u_{i,1}\right)h_1 =\left(I +  \sum_{i=1}^q A_i u_{i,1}\right) \left(I+ \sum_{i=1}^q A_i u_{i,0}\right)h_{0}.
\end{align*}
Iterating, the claim follows.
\end{proof}

We consider a cost function of the form 
\begin{equation}\label{costf}
\mathcal{E}(u) = F(h_T) + \sum_{t=0}^{T-1} L_{t}(h_{t},u_{t}).
\end{equation}
where $h = h(u)$ denotes the solution of \eqref{diffh} for a given control $u = (u_t)$, for a differentiable terminal cost function $F: \mathcal{H} \to \mathds{R}$ and a differentiable time-dependent cost function $L_t:  \mathcal{H}\times \mathds{R}^{q} \to \mathds{R}$. 


The goal is to learn controls $u$ by solving
\begin{equation}\label{minprob}
\min_{u \in \mathds{R}^{q \times T}} \mathcal{E}(u).
\end{equation} 
We notice that \eqref{costf} is not convex in $u$ in general, even if $F$ and $L_t$ are, due to the fact that $h$ is not. The following counter-example illustrates that, even in a simple one-dimensional case under a convex terminal cost, convexity of the mapping $u \mapsto \mathcal{E}(u)$ is unwarranted.
\begin{example}\label{exconv} Consider $\mathcal{H} = \mathds{R}$, $q=1$, $T=2$, and $Ah = ah$, for some $a \in \mathds{R}$. Take a quadratic terminal cost, $F(h) = h^2$.

It follows from Lemma \ref{lemmasolh} that, at terminal time $T= 2$, $h$ is given by $h_2 = (1+au_1)(1+au_2)h_0$, so that $\mathcal{E}(u) = F(h_2) := h_2^2 = (1+au_1)^2(1+au_2)^2h^2_0$, whose Hessian matrix is
\[
D^2 \mathcal{E}(u) = 
\begin{pmatrix}
2a^2(1+au_2)^2h^2_0 & 4a^2(1+au_1)(1+au_2)h^2_0 \\
4a^2(1+au_1)(1+au_2)h^2_0 & 2a^2(1+au_1)^2h^2_0
\end{pmatrix}.
\]
When evaluated at $u_1 = u_2 = 0$ the Hessian is 
\[
D^2 \mathcal{E}(0) = 2a^2 h^2_0
\begin{pmatrix}
1 & 2 \\
2 & 1
\end{pmatrix},
\]
whose determinant is negative. This shows that the cost function is in general not convex in $u$.
\end{example}

Given the non convexity of $\mathcal{E}(u)$, one can hope at best for local minimisers, which we will search with gradient-based methods. Thereto, we now derive a general expression for the gradient.

We define the adjoint response $h^*_t \in \mathcal{H}$ along the solution $h_t$ for the control system \eqref{diffh} with control $(u_t)$ and cost function $\mathcal{E}$ as the solution to the linear terminal value problem
\begin{nalign} \label{der1}
\Delta h_{t+1}^* &= -A^*[u_t]h_{t+1}^* - D_1L_t(h_t, u_t),\\
h^*_T &= DF (h_T).
\end{nalign}
where $A^*$ is the adjoint operator of $A$.


\begin{theorem}\label{lemmagrad} The gradient of the cost function with respect to the controls $(u_t)$ is given by
\begin{equation}\label{der2}
D\mathcal{E}(u)_{i,t} = \langle h_{t+1}^*,A_ih_{t} \rangle_{\mathcal{H}} +  D_{2,i}L_t(h_t, u_{t}),
\end{equation}
where $D\mathcal{E}(u) \in \mathds{R}^{q\times T}$ $\forall u \in \mathds{R}^{q\times T}$.
\end{theorem}
\begin{proof}
Define the residual cost incurred from time $t$ to the last period $T$ as the function 
\begin{align*}
J_t &: \mathcal{H} \times \mathds{R}^{q\times (T-t)} \to \mathds{R},\\
J_T(h_T) &: = F(h_T),\\
J_t(h_t, u_t, \dots, u_{T-1}) &:=  L_t(h_t,u_t) + J_{t+1}\left(h_{t+1}, u_{t+1}, \dots, u_{T-1}\right), 
\end{align*}
recalling that $h_{t+1}$ can be expressed as a function of $h_t$ by means of the relation $h_{t+1} = (I+A[u_t])h_t$.
Notice that 
\begin{equation}\label{j0}
\mathcal{E}(u) = J_0(h_0,u).
\end{equation}
We can obtain a recursive relation for the gradient of the residual cost at time $t$ with respect to the control $u_t$. By the chain rule, this is given by
\begin{align*}
D_{u_{t}}J_t(h_t, u_t, \dots, u_{T-1}) & = D_{2} L_t(h_t,u_t) +  \langle D_{h_{t+1}} J_{t+1}((I+A[u_t])h_t , u_{t+1}, \dots, u_{T-1}), \,D_{u_{t}}h_{t+1} \rangle_{\mathcal{H}}\\
& =  D_{2} L_t(h_t,u_t) +  \langle D_{h_{t+1}} J_{t+1}\left(h_{t+1}, u_{t+1}, \dots, u_{T-1}\right), \,Ah_{t} \rangle_{\mathcal{H}}.
\end{align*}
For a generic time $s= t, \dots, T-1$, the gradient can again be computed recursively, as follows:
\[
D_{u_{s}}J_t(h_t, u_t, \dots, u_{T-1}) =  
\begin{cases}
D_{2} L_t(h_t,u_t) +  \langle D_{h_{t+1}} J_{t+1}, \,Ah_{t} \rangle_{\mathcal{H}},\quad & s = t \\
D_{u_s} J_{t+1}((I+A[u_t])h_t , u_{t+1}, \dots, u_{T-1}), \quad & s = t+1, \dots, T-1.\\
\end{cases}
\]
This implies that, for any $t=0, \dots, T-1$:
\[
D_{u_{s}}J_t(h_t, u_t, \dots, u_{T-1}) =  
\begin{cases}
D_{2} L_t(h_t,u_t) +  \langle D_{h_{t+1}} J_{t+1}\left(h_{t+1}, u_{t+1}, \dots, u_{T-1}\right), \,Ah_{t} \rangle_{\mathcal{H}},\quad & s = t \\
D_{2} L_{t+1}(h_{t+1},u_{t+1}) +  \langle D_{h_{t+2}} J_{t+2}, \,Ah_{t+1} \rangle_{\mathcal{H}},\quad & s = t+1 \\
\quad \vdots & \, \vdots \\
D_{2} L_{T-1}(h_{T-1},u_{T-1}) +  \langle D_{h_{T}} J_{T}, \,Ah_{T-1} \rangle_{\mathcal{H}},\quad & s = T-1.\\
\end{cases}
\]
We now consider the derivative of the residual cost with respect to $h$, namely the linear operator $D_{h_{t}} J_{t}(h, u): \mathcal{H} \to \mathds{R}$. In order to show that $D_{h_{t}} J_{t}(h_t, u_t, \dots, u_{T-1}) = h^*_t$, we proceed by induction.

At time $T$ the derivative of the residual cost is equal to the derivative of the terminal cost:
\[
D_{h_T}J_T(h_T)  = DF(h_T) = h^*_T,\\
\]
with the second equality following from the definition of the adjoint equation.
As for $t<T$, one has:
\begin{align*}
D_{h_t}J_t(h_t, u_t, \dots, u_{T-1}) &= D_{1} L_{t}(h_t,u_t) + D_{h_t}J_{t+1}\left((I+A[u_t])h_t, u_{t+1}, \dots, u_{T-1}\right) \\
 & = D_{1} L_{t}(h_t,u_t) \\
 & \quad + (I+A^*[u_t]) D_{h_{t+1}}J_{t+1}\left(h_{t+1}, u_{t+1}, \dots, u_{T-1}\right) |_{h_{t+1} = (I + A[u_t])h_t} \\
& = D_{1} L_{t}(h_t,u_t)  +  (I+A^*[u_t])h^*_{t+1} \\
& =  h^*_{t},
\end{align*}
where the third equality follows from the induction step and the last one from the definition of the adjoint equation.

The statement of the theorem follows because of \eqref{j0}.
\end{proof}


The following example shows an alternative choice for the operators $A$ with respect to that implemented in this work. In its generality, our method can encompass operator choices which mirror the common neural network architectures, and can hence consist of the composition of linear and non-linear functions, as explained below.

\begin{example}
\subsubsection*{Neural-network-like Operator.}
Fix functions $\varphi_1, \dots, \varphi_q: \mathcal{X} \to \mathcal{X}$ such that $h \circ \varphi \in \mathcal{H}$ $\forall h \in \mathcal{H}$.

As an example, and inspired by the architecture of a neural network, one can take the composition of a linear and a non-linear operator, so $\varphi_i(\xi) = \sigma_i(W_i \xi + b_i)$. Here, $W_i: \mathcal{X} \to \mathcal{X}$ are linear functions and $\sigma_i: \mathcal{X} \to \mathcal{X}$ are non-linear functions. 

Assume that there exists a subset $\mathcal{X}_0 \subseteq \mathcal{X}$, dense in $\mathcal{X}$ in the sense that $h(\xi) = 0$ implies that $h=0$ $\forall \xi \in \mathcal{X}$, and such that the set $\mathcal{S} = \{k_{\xi} \, \lvert \, \xi \in \mathcal{X}_0\}$ is linearly independent in $\mathcal{H}$. Define the operator $A_i$ on $\mathcal{S}$
\[
A_i k_{\xi}(\cdot):= k_{\varphi_i(\xi)}(\cdot).
\]
The adjoint operator $A_i^*$ of $A_i$ is defined for any $h \in \mathcal{H}$ and we have
\[
\langle A_i^*h, k_{\xi} \rangle_{\mathcal{H}} = \langle h, A_i k_{\xi} \rangle_{\mathcal{H}} = \langle h, k_{\varphi_i(\xi)} \rangle_{\mathcal{H}} = h(\varphi_i(\xi)) =  \langle h \circ \varphi_i, k_{\xi} \rangle_{\mathcal{H}}.
\]
In addition, $\mathcal{X}_0$ is dense in the sense that $h(x) = 0$ implies that $h=0$ $\forall x \in E$.
%
%
We consider functions $h_t$ solving the difference equation \eqref{diffh} under the choice of operators $A_1, \dots A_q$, with the additional constraint that $h_0 \in \mathcal{D} = \emph{span} \, \mathcal{S}$. 

As an example, when the initial condition is given by $k_{\xi_0}$, Lemma \ref{lemmasolh} shows that function $h_t$ will be given by
a linear combination of compositions of the operators $(A_i)_i$ applied to $k_{\xi_0}$. 

Notice that
\[
A_iA_jk_{\xi_0} = k_{\varphi_i (\varphi_j (\xi_0))},
\]
and
\[
(A_i)^p k_{\xi_0} = k_{\underbrace{\varphi_i \circ  \dots  \circ \varphi_i}_{\text{p times}}(\xi_0)},
\]
so that, more specifically, $h_t$ will be a linear combination of the elements $k_{\xi_0}$, $k_{\varphi_i(\xi_0)}$, $k_{\varphi_i \circ \varphi_j(\xi_0)}$, ..., $k_{\varphi_{i_1} \circ\varphi_{i_2} \circ \dots \circ \varphi_{i_t}(\xi_0)}$, for $(i_1, i_2 \dots, i_t) \in \{ 1, \dots, q\}^t$. The cardinality of this set is given by $1+ q + \dots + q^t = \frac{q^{t+1}-1}{q-1}$. Consider as an example $q = 10$ operators, the number of basis functions will scale approximately like $10^T$. 
\end{example}

\subsection{Relation to Pontryagin's Maximum Principle}
We show in the following that the expression obtained for the gradient of the total cost is consistent with Pontryagin's maximum principle.

We treat the minimisation problem \eqref{minprob} as a constrained optimisation problem and study its Lagrangian. The constraint is given by the control system \eqref{diffh}. Hence, the Lagrangian is
\begin{equation} \label{lagr}
\mathcal{L}(h,u,\lambda) = F(h_T) +\sum_{t=0}^{T-1} \left( L_t(h_t,u_t) -\langle h_{t+1}-(1+A[u_t])h_t, \lambda_{t+1}\rangle_{\mathcal{H}} \right),
\end{equation}
for lagrangian multipliers $\lambda_t \in \mathcal{H}$. We denote as
\begin{equation} \label{hamil}
\notag
H_t(h,u,\lambda) = \langle A[u]h,\lambda \rangle_{\mathcal{H}} + L_t(h,u)
\end{equation}
the Hamiltonian of the system, which allows us to rewrite the Lagrangian \eqref{lagr} as 
\begin{align*}
\mathcal{L}(h,u,\lambda) & = F(h_T)- \langle h_{T}, \lambda_{T}\rangle_{\mathcal{H}}+ \langle h_{0}, \lambda_{0}\rangle_{\mathcal{H}} \\
& \quad + \sum_{t=0}^{T-1} \left( H_t(h_t,u_t,\lambda_{t+1}) +  \langle h_{t}, \lambda_{t+1}-\lambda_{t}\rangle_{\mathcal{H}}\right).
\end{align*}
The derivatives of the Lagrangian can be computed as:
\begin{align*}
D_{h_t}\mathcal{L}(h,u,\lambda)  & = D_1H_t(h_t,u_t,\lambda_{t+1}) + \lambda_{t+1}-\lambda_t,\\
D_{h_T}\mathcal{L}(h,u,\lambda)  & = DF(h_T) - \lambda_T,\\
D_{u_{i,t}}\mathcal{L}(h,u,\lambda) &= D_{2,i} H_t(h_t,u_t,\lambda_{t+1}),\\
D_{\lambda_{t+1}}\mathcal{L}(h,u,\lambda)  & = D_3H_{t}(h_{t},u_{t},\lambda_{t+1})+ h_{t}-h_{t+1},\\
D_{\lambda_0}\mathcal{L}(h,u,\lambda)  & = 0.
\end{align*}
This leads to the first-order conditions
\begin{align}
\lambda_{t+1}-\lambda_t & = - D_1H_t(h_t,u_t,\lambda_{t+1}),  \label{foclagr1} \\
\lambda_T &= DF(h_T), &  \label{foclagr2}\\
D_{2,i}H_t(h_t,u_t,\lambda_{t+1}) &= 0. \label{fochamilt}
\end{align}

This is remarkable, as it amounts to saying that one could formulate Pontryagin maximum principle as a corollary of Theorem \ref{lemmagrad}. Indeed, the RHS of \eqref{der1} is equal to $-D_1 H(h,u,\lambda)$, so that \eqref{foclagr1} and \eqref{foclagr2} correspond to the system of the adjoint equation \eqref{der1}. In addition, the RHS of \eqref{der2}, the gradient of the cost function, is equal to $D_2 H(h,u,p)$, the LHS of \eqref{fochamilt}.

Stated differently, the conditions of optimality ($i$) retrieve the adjoint difference equation and its terminal condition as Lagrangian multipliers of the constrained optimisation problem, and ($ii$) demand that the gradient of the Hamiltonian with respect to $u$ -- which turns out to be the gradient of the total cost with respect to the controls -- be zero. However, it has to be stressed that, in general, $D\mathcal{E}(u)$ is not equal to $D_2\mathcal{L}(h,u,p)$, even though at the optimum control the two gradients are equal, and equal to zero.

Notice that, in principle, one could handle problem \eqref{minprob} by directly solving the first-order conditions \eqref{foclagr1}-\eqref{fochamilt}. However, this system of equations has no explicit solution and would still entail solving for $h$ and $h^*$ backward and forward. Eventually, one would start by an initial guess for the controls $u$ and perform a gradient search in order to bring the Hamiltonian closer to its minimum. However, this approach would result in the same optimisation method which follows from minimising the total cost by shifting $u$ by its gradient \eqref{der2}.

\section{Workable Specification}\label{ws}
Since the control system \eqref{diffh} is in infinite dimension, in order to make it implementable in practice we now propose a workable specification that renders it amendable for the kernel trick. We let $\mathds{M}$ denote some probability measure on the state space $\mathcal{X}$ and consider the embedding $J_{\mathds{M}}: \mathcal{H} \to L^2_{\mathds{M}}$ mapping a function into its $\mathds{M}$-a.s.-equivalence class. 
The adjoint of $J_{\mathds{M}}$ is $J^*_{\mathds{M}}: L^2_{\mathds{M}} \to \mathcal{H}$ and acts on $g \in L^2_{\mathds{M}}$ as follows:
\begin{equation} \label{adjop}
\notag
J^*_{\mathds{M}}g = \int_{\mathcal{X}} g(\xi) k(\xi,\cdot) \mathds{M}(d \xi).
\end{equation}

The operators $A_i$ are assumed to factorise as follows:
\[
A_i = J^*_{\mathds{M}}B_i J_{\mathds{M}}:  \mathcal{H} \to L^2_{\mathds{M}} \to L^2_{\mathds{M}} \to  \mathcal{H},
\]
\begin{center}
\[
\begin{CD}
A_i: \, &  \mathcal{H} @>>>   \mathcal{H}\\
&@VVJ_{\mathds{M}}V @AAJ_{\mathds{M}}^*A\\
&L^2_{\mathds{M}}  @>B_i>> L^2_{\mathds{M}} .
\end{CD}
\]
\end{center}
for some bounded operators $B_i$ on $L^2_{\mathds{M}}$.

The control system \eqref{diffh} in $\mathcal{H}$ can be pulled down to a control system in $L^2_{\mathds{M}}$ by  observing that $g_t = J_{\mathds{M}}h_t$ solves:
\begin{align} \label{system_g}
\Delta g_{t+1} &= J_{\mathds{M}} J_{\mathds{M}}^{*}B[u_t]g_{t},\\
\notag g_{0} &= J_{\mathds{M}} \phi_0.
\end{align}

The evaluation of the function $h$ follows from \eqref{system_g} and the definition of $g$. Concretely,
we can express $h_t$ in terms of $g_t$ as
\begin{nalign}\label{hsol}
h_t & = \phi_0 + \sum_{s=0}^{t-1} A[u_s]h_s = \phi_0 + \sum_{s=0}^{t-1} J_{\mathds{M}}^{*}B[u_s]g_{s} \\
&=  \phi_0 +  \sum_{s=0}^{t-1} \int_{\chi} k(y,\cdot) B[u_s]g_s(y) \mathds{M}(dy).
\end{nalign}

For the adjoint response we define $g_t^* = J_{\mathds{M}}h^*_t$ and infer from  \eqref{der1} that it solves
\begin{nalign} \label{system_gstar}
\Delta g^{*}_{t+1} &= -J_{\mathds{M}} J_{\mathds{M}}^{*}B^*[u_t]g^*_{t+1} - J_{\mathds{M}} D_1 L_t(h_t,u_t),\\
 g_T^* &= J_{\mathds{M}}DF(h_T).
\end{nalign}
We obtain the following expression for the derivative of the cost function \eqref{der2} in terms of $g_t$ and $g_t^*$:
\begin{equation}\label{grg}
D\mathcal{E}(u)_{i,t} = \langle g_{t+1}^*, B_i g_t \rangle_{L^2_{\mathds{M}}} + D_{2,i}L_t(h_t, u_t).
\end{equation}
Since we need to query the solution of \eqref{system_gstar} for many terminal conditions, it is more efficient to solve such system for a set of basis functions $\Psi$ (operators on $L^2_{\mathds{M}}$) and express $g^*$ in terms of these functions. The following Lemma \ref{lemmagstar} shows how an expression for $g^*$ can be derived.

\begin{lemma}\label{lemmagstar}
Let the linear operator $\Psi_t: L^2_{\mathds{M}} \to L^2_{\mathds{M}}$ satisfy the system of backward difference equations
\begin{nalign}\label{system_bigPsi}
  \Delta \Psi_{t+1} &= b_t \Psi_{t+1}, \\ 
  \Psi_T &= I_{L^2_{\mathds{M}}}.
\end{nalign}
In addition, let $I_{L^2_{\mathds{M}}}$ denote the identity operator on $L^2_{\mathds{M}}$,
\[
b_t = -J_{\mathds{M}} J_{\mathds{M}}^{*}B^*[u_t],
\]
and
\[
l_t = -J_{\mathds{M}} D_1 L_t(h_t,u_t).
\]
Then, the solution $g^*$ of system \eqref{system_gstar} is given by
\begin{equation}\label{gfrompsi}
g^*_t = \Psi_t f_t =  \Psi_t g^*_T -  \Psi_t \sum_{s=t}^{T-1} \Psi_s^{-1} l_s.
\end{equation}
\end{lemma}
\begin{proof}
Notice to begin with that system \eqref{system_bigPsi} has solution
\[
\Psi_t =(I_{L^2_{\mathds{M}}}-b_t) \Psi_{t+1} =  \prod_{s=t}^{T-1}(I_{L^2_{\mathds{M}}}-b_s).
\]
We define the auxiliary function $ f_t = \Psi_t^{-1} g^*_t$, which satisfies the difference equation
\begin{align}
  \Delta f_{t+1} &= \Psi_{t+1}^{-1} g^*_{t+1} - \Psi_{t}^{-1} g^*_{t} \nonumber\\
  & =  \Psi_{t+1}^{-1} g^*_{t+1} - \Psi_{t}^{-1} g^*_{t} + \Psi_{t}^{-1} g^*_{t+1}- \Psi_{t}^{-1} g^*_{t+1}\nonumber \\
  & = \left( \Psi_{t+1}^{-1}- \Psi_{t}^{-1} \right)g^*_{t+1} + \Psi_{t}^{-1} (g^*_{t+1}-g^*_{t}) \nonumber\\
  & = \left( \Psi_{t+1}^{-1}- \Psi_{t}^{-1} \right)g^*_{t+1} + \Psi_{t}^{-1} (b_t g^*_{t+1}+l_t)\nonumber\\
  & =  \left( \Psi_{t+1}^{-1}- \Psi_{t}^{-1} (I- b_t)  \right)g^*_{t+1} - \Psi_{t}^{-1}l_t\nonumber\\
  & = \Psi_{t}^{-1}l_t \label{lasteq},
\end{align}
along with terminal condition $ f_T = g^*_T$. Notice that the equality \eqref{lasteq} follows from the fact that $ \Psi_{t+1}^{-1} =  \Psi_{t}^{-1} (I- b_t)$.

The solution to the system of backward difference equations for $f$ is given by
\[
f_t = g^*_T -\sum_{s=t}^{T-1} \Psi_s^{-1} l_s,
\]
which implies that $g^*$ can be computed directly from $\Psi$ as in \eqref{gfrompsi}.
\end{proof}
Hence, we can solve the system of difference equations for the basis functions $\Psi$ and query $g^{*}$ for several terminal conditions using such basis functions instead of having to solve the system for $g^*$ repeatedly for each of the terminal conditions.

We consider in the following examples of operators which can be used in applications, and study some of their properties.
\begin{example}\label{exb1}
\subsubsection*{Diagonal Operator.}
We consider $g  \in L_{\mathds{M}}^2$. We define $B$ as a diagonal operator, so that $B g = b g$, for some $b \in L_{\mathds{M}}^{\infty}$. Then, it holds that
\[
\|B g\|_{L_{\mathds{M}}^2}^2 = \int_{\mathds{R}^d} b(\xi)^2 g(\xi)^2 \mathds{M}(d\xi) \leq \|b\|^2_{L_{\mathds{M}}^{\infty}}\|g\|_{L_{\mathds{M}}^2}^2,
\]
with equality for $g=1$. It follows that the operator norm is
\[
\|B\| =  \|b\|_{L_{\mathds{M}}^{\infty}}.
\]
\end{example}
\begin{example}\label{exb2}
\subsubsection*{Rank-one Operator.}
We now consider the operator $B$ to be rank-one, namely such that for some $\beta \in L^2_{\mathds{M}}$
\[
B= \beta \otimes \beta.
\]
The operator $B$ acts on $g \in L^2_{\mathds{M}}$ as follows:
\[
B g = \langle g,\beta \rangle_{L^2_{\mathds{M}}} \beta.
\]
Hence:
\[
\|Bg\|_{L_{\mathds{M}}^2} = \|\beta\|_{L_{\mathds{M}}^2}\lvert\langle \beta,g \rangle_{L_{\mathds{M}}^2}\rvert \leq \|\beta\|_{L_{\mathds{M}}^2}^2\|g\|_{L_{\mathds{M}}^2},
\]
with equality for $g = \beta$. It follows that the operator norm is
\[
\|B\| =  \|\beta\|_{L_{\mathds{M}}^{2}}^2.
\]
\end{example}

\subsection{Empirical Workable Specification} 
We consider a sample $(\xi_1, \dots, \xi_m)$ drawn from $\mathds{M}$, where we assume that $\xi_i \neq \xi_j$ $\forall i \neq j$. Then, when replacing $\mathds{M}$ by its empirical counterpart 
$$\widehat{\mathds{M}} = \frac{1}{m} \sum_{i=1}^m \delta_{\xi_i},$$
the space $L^2_{\widehat{\mathds{M}}}$ becomes $m$-dimensional and we can identify functions $g \in L^2_{\widehat{\mathds{M}}}$ with vectors $g = (g(\xi_1), \dots,  g(\xi_m))^\top$. We represent inner products $\langle f,g\rangle_{L^2_{\widehat{\mathds{M}}}}= \frac{1}{m} f^\top g$ and consider the operators $J_{\widehat{\mathds{M}}}$ and $J^*_{\widehat{\mathds{M}}}$. These act on $h$ and $g$, respectively, as 
\[
J_{\widehat{\mathds{M}}}h = (h(\xi_1), \dots, h(\xi_m))^\top
\]
and
\[
J^*_{\widehat{\mathds{M}}}g = \frac{1}{m} \sum_{i=1}^m g(\xi_i) k(\xi_i,\cdot),
\]
where the kernel function gives rise to the $m\times m$ kernel matrix $K(i,j) =k(\xi_i,\xi_j)$, which is related to the operator $J_{\widehat{\mathds{M}}} J_{\widehat{\mathds{M}}}^{*}g = \frac{1}{m} K g$. The controlled linear difference equations \eqref{system_g} become $m$-dimensional.

\begin{example} The empirical counterpart to the diagonal operator introduced in Example \ref{exb1}, is
\[
\widehat{B}g = b \circ g,
\]
where the $\circ$ denotes the Hadamard product. As for the empirical counterpart to $A$, one has
\[
\widehat{A}h = \frac{1}{m} \sum_{i=1}^m b(\xi_i)g(\xi_i)k_{\xi_i}.
\]
\end{example}

%


\begin{example} The empirical counterpart to the diagonal operator introduced in Example \ref{exb1}, is
\[
\widehat{B}g = \frac{1}{m} \beta \beta^\top g.
\]
As for the empirical counterpart to $A$, one has
\[
\widehat{A}h = \frac{1}{m} \sum_{i=1}^m \left(\frac{1}{m} \sum_{j=1}^m b(\xi_j) g(\xi_j) \right) b(\xi_i)k_{\xi_i}.
\]
\end{example}



Equation \eqref{hsol} shows that the values of the function $g$ at given points $\xi_j$, $j = 1, \dots, m$, allow one to evaluate the function $h$ at $x$ as follows:
\begin{flalign*}
h_{t}(x)& = \phi_0(x) + \sum_{s=0}^{t-1}\langle k(x,\cdot),B[u_s]g_{s}\rangle_{L^2_{\widehat{\mathds{M}}}} = \phi_0(x) + \frac{1}{m}\sum_{s=0}^{t-1} \sum_{j=1}^m k(x,\xi_j) (B[u_s]g_{s})(\xi_j)  ,
\end{flalign*}
which has more compact representation:
\[
h_{t} = \phi_0+ \frac{1}{m} k(\cdot,\xi^T)\sum_{s=0}^{t-1}  (B[u_s]g_{s}).
\]
As a result, notice that our workable specification gives terminal solutions of the form
\begin{equation} \label{base}
h_T = \sum_{j=1}^m c_j  k(\cdot,\xi_j).
\end{equation}

%
%


\section{Optimisation Algorithms} \label{opti}
We propose three different algorithms to solve the minimisation problem \eqref{minprob}, which is in general not convex as shown in Example \ref{exconv}. Algorithm \ref{alg1} is a traditional gradient-descent method and is presented in Section \ref{gradient_descent}. Having derived an expression for the gradient of the error with respect to the controls $u$, and having set an initial guess for the controls, the controls are shifted by a multiple of the (negative) gradient at each iteration.

Algorithm \ref{alg2} is described in Section \ref{least_squares} and involves a set of  first-order approximations of the functions $h_t$, $t = 1, \dots, T$. The controls are updated iteratively so as to minimise the cost function for the linear approximation of $h_t$.

Algorithm \ref{alg3} in Section \ref{ilssec} extends Algorithm \ref{alg2} by introducing a regression step at the end. The output of the algorithm is then the linear approximation of $h(x)$.

We first present cost functions which are often used in applications -- for regression and classification tasks (Section \ref{cosf}), and then describe the algorithms used to minimise said cost functions with respect to the controls.


\subsection{Cost Functions}\label{cosf}
Henceforth we assume that the cost functions are expected functions of the form
\begin{nalign} \label{costexp}
F(h) &= \mathds{E}_{\mathds{P}}[l_T(h(X),Y)],\\
L_t(h,u) &= \mathds{E}_{\mathds{P}}[l_t(h(X),Y,u)],
\end{nalign}
where $l_T$ and $l_t$ are differentiable mappings, $l_T: \mathds{R} \times \mathds{R} \to \mathds{R}$ and $l_t: \mathds{R} \times \mathds{R}  \times \mathds{R}^q \to \mathds{R}$. In addition, in \eqref{costexp},  $X$ and $Y$ denote random variables relative to a measure $\mathds{P}$, $X$ belonging to the state space $\mathcal{X}$ and $Y$ being a target which depends on the application, as clarified by the following examples.
\begin{example}{\textbf{Regression.}} In the case of a regression $Y = f(X)$ for some target function $f \in L^2_{\mathds{P}}$ and $l_T = (h(X)-Y)^2$.
This corresponds to the typical $L^2$ cost which results in the terminal cost function given by
\begin{equation} \label{cost1}
F(h) =  \|h-f\|^2_{L^2_{\mathds{P}}}.
\end{equation}
In the same spirit, one can consider a time-dependent cost function such as
\begin{equation} \label{cost2}
L_t(h,u)= \|h-f_t\|^2_{L^2_{\mathds{P}}} + \lambda \|u\|^2_{\mathds{R}^q}.
\end{equation}
For the cost functions \eqref{cost1} and \eqref{cost2} considered above we obtain the following expressions for the derivatives that appear in \eqref{der1} and \eqref{der2}, 
\begin{align}
DF (h) &= 2 \, J_{\mathds{P}}^*(J_{\mathds{P}}h - f),\label{dfh} \\
D_1L_t(h, u) &= 2\, J_{\mathds{P}}^*(J_{\mathds{P}}h - f_t),\notag \\
D_2L_t(h, u) &= 2 \, \lambda u,\notag
\end{align}
where $J_{\mathds{P}} : \mathcal{H} \mapsto L^2_{\mathds{P}}$ denotes the canonical embedding, which maps functions $h$ into their $\mathds{P}$-a.s.-equivalence class.
\end{example}

\begin{example}\label{logregex} \textbf{Logistic Regression.} In the case of a logistic regression $Y \in \{0,1\}$ and the terminal cost function is the cross entropy
\begin{equation}\label{crosse0}
l_T(h(X),Y) = \log{\left(1+e^{h(X)}\right)} - Y h(X).
\end{equation}
This leads to an expected terminal cost of the form
\begin{equation}\label{crosse}
    F(h) = \mathds{E} \left[\log{\left(1+e^{h(X)}\right)} - Y h(X)\right].
\end{equation}
Having modelled the probability $p$ of an observation having label equal to 1 by means of the sigmoid function 
\[
\sigma(x) = \frac{e^{h(X)}}{1+e^{h(X)}},
\]
equation \eqref{crosse} arises from the negative log likelihood
\begin{flalign*} 
  - \mathds{E} \left[ Y \log p(X) + (1-Y)\log (1-p(X)) \right]  
& = - \mathds{E} \left[ Y \log \sigma(X) + (1-Y)\log (1-\sigma(X)) \right] \\
& = \mathds{E} \left[ \log(1+e^{h(X)}) - Y h(X) \right].
\end{flalign*}
The gradient of $F$ is again denoted as $DF(h) \in \mathcal{H}$ and is given by
\begin{equation}\label{gradlogi}
DF(h) = J^* \left(\frac{e^{h}}{1+e^{h}}- p\right).
\end{equation}
To derive expression \eqref{gradlogi}, we consider how the gradient acts on $v \in \mathcal{H}$.
\begin{nalign}
\notag
    DF(h_T)v &= \frac{d}{d \epsilon} F(h_T + \epsilon v)\rvert_{\epsilon = 0} \\
    & = \mathds{E} \left[\frac{v(X) e^{h(X)}}{1+e^{h(X)}} - Y v(X) \right] \\
\end{nalign}
As far as the term
\[
\mathds{E} \left[\frac{v(X) e^{h(X)}}{1+e^{h(X)}} \right] = \mathds{E} \left[v(X) \sigma(h(X)) \right]
\]
is concerned, we can rewrite it as
\[
\mathds{E} \left[v(X) \sigma(h(X)) \right] = \langle \sigma(h), Jv \rangle_{L^2_{\mathds{P}_X}} = \langle J^* \sigma(h), v, \rangle_{\mathcal{H}}.
\]
We now consider the term
\begin{nalign}\label{trentadue}
    \mathds{E} \left[ Y v(X) \right]& = \int_{\mathcal{X} \times \{0,1 \}} y \, v(x) \, \mathds{P}_{X,Y} (dx,dy)\\
    & = \int_{\mathcal{X}} 1 \,  v(x) \mathds{P}_{X\lvert Y = 1}(dx) \mathds{P}(Y=1).
\end{nalign}
Now, Bayes' rule yields
\[
 \mathds{P}_{X\lvert Y = 1}(dx) =  \frac{\mathds{P}_{Y \lvert X = x}(\{1\})}{\mathds{P}(Y=1)}  \mathds{P}_{X}(dx) = \frac{p(x)}{\mathds{P}(Y=1)} \mathds{P}_{X}(dx),
\]
so that the previous expression \eqref{trentadue} can be rewritten as
\begin{flalign*} 
\mathds{E} \left[ Y v(X) \right]& =\int_{\mathcal{X}} v(x) \, p(x) \, \mathds{P}_{X}(dx) = \langle p, Jv, \rangle_{L^2_{\mathds{P}_X}} \\
 & = \langle J^*p, v, \rangle_{\mathcal{H}}.
\end{flalign*}
Hence, we can conclude that the gradient of the cost function is 
\[
DF(h) = J^* \left(\sigma(h)- p\right).
\]
\end{example}

In practice, we observe training samples $(x_1, \dots , x_n)$ drawn from $\mathds{P}$. This leads to the empirical measure 
$$\widehat{\mathds{P}} = \frac{1}{n} \sum_{i=1}^n \delta_{x_i},$$
 and provides finite-sample estimates of the cost functions $\hat{F}$ and $\hat{L}_t$ which arise by replacing the measure $\mathds{P}$ by its empirical counterpart $\mathds{\hat{P}}$ in \eqref{costexp}. Similarly, we write $\hat{\mathcal{E}}(u)$ and $\widehat{D\mathcal{E}}(u)$ to refer to the empirical versions of the cost function $\mathcal{E}(u)$ and of its gradient $D\mathcal{E}(u)$.

\begin{example} The empirical counterpart to the derivative of the terminal cost function \eqref{dfh} in the case of the regression is
\[
\widehat{DF}(h)  = \frac{2}{n} \sum_{i=1}^n k(x_i, \cdot)(h(x_i)-f(x_i)).
\]
\end{example}
 
\begin{example} Having approximated $p(x)$ by its empirical counterpart $\hat{p}(x)$, the empirical counterpart to the derivative of the terminal cost function \eqref{gradlogi} in the case of the logistic regression is
\begin{flalign*} 
\widehat{DF}(h) & = \frac{1}{n} \sum_{i=1}^n k(x_i, \cdot) \left(\frac{e^{h(x_i)}}{1+e^{h(x_i)}} - \hat{p}(x_i) \right)  \\
& = \frac{1}{n} \sum_{i=1}^n k(x_i, \cdot) \left(\frac{e^{h(x_i)}}{1+e^{h(x_i)}} - \frac{1}{\#\{j: x_j = x_i\}} \sum_{j: x_j = x_i} y_j \right) \\
& = \frac{1}{n} \sum_{i=1}^n k(x_i, \cdot) \left(\frac{e^{h(x_i)}}{1+e^{h(x_i)}} - y_i \right).
\end{flalign*}
 \end{example}

The following subsections are devoted to the detailed description of the optimisation strategies used in the search for the optimal controls $(u_t)$.

\subsection{Gradient Descent} \label{gradient_descent}

Algorithm \ref{alg1} reports the steps implemented in a gradient descent method. After an initial guess for the controls, these are shifted by a multiple $\alpha \in \mathds{R}$,  $\alpha > 0$, of the negative gradient at each iteration. 
\IncMargin{0mm}
\begin{algorithm}
\SetAlgoLined
\vspace{2mm}
\hspace{-5mm} \textbf{Output:} approximating functions $\left(h_1\left(u^{(i)}\right), \dots, h_T\left(u^{(i)}\right)\right)$.\\

 \SetKwBlock{Begin}{Initialise}{end}
 \Begin{
$u^{(0)} \leftarrow$ initial guess for the controls\\
$\alpha \leftarrow $ learning rate\\
$n \leftarrow$ size of the mini batch\\
$T \leftarrow$ number of time steps\\
$i \leftarrow 0$\\
}
\SetKwBlock{Begin}{While stopping criterion not satisfied}{end}
 \Begin{
Sample points of the mini batch $x_1^{(i)}, \dots,  x_n^{(i)}$\\ 
$\widehat{D\mathcal{E}}(u^{(i)}) \leftarrow $ gradient of the error with respect to the controls computed over the mini batch\\
$u^{(i+1)} \leftarrow u^{(i)} -  \alpha \widehat{D\mathcal{E}}(u^{(i)})$\\
$i \leftarrow i + 1$\\
}
\textbf{Return} $\left(h_1\left(u^{(i)}\right), \dots, h_T\left(u^{(i)}\right)\right)$
\vspace{2mm}
 \caption{Stochastic Gradient Descent}
 \label{alg1}
\end{algorithm}

\subsection{Iterative Regression} \label{least_squares}

\IncMargin{0mm}
\begin{algorithm}
\SetAlgoLined
\vspace{2mm}
\hspace{-5mm} \textbf{Output:} approximating functions $\left(h_1\left(u^{(i)}\right), \dots, h_T\left(u^{(i)}\right)\right)$.\\

 \SetKwBlock{Begin}{Initialise}{end}
 \Begin{
$u^{(0)} \leftarrow$ initial guess for the controls\\
$F \leftarrow $ terminal cost function\\
$L \leftarrow $ time-dependent cost function\\
$T \leftarrow$ number of time steps\\
$n \leftarrow$ size of the mini batch\\
$\lambda \geq 0 \leftarrow$ regularisation coefficient\\
$i = 0$\\
}
\SetKwBlock{Begin}{While stopping criterion is not satisfied}{end}
 \Begin{
Sample points of the mini batch $x_1^{(i)}, \dots,  x_n^{(i)}$\\
Solve the following problem on the mini batch:
\begin{flalign*}
\beta^{(i)} \leftarrow & \mbox{argmin}_{\beta \in \mathds{R}^{q \times T}} \Biggl\{\hat{F}\left(h_T(u^{(i)}) + D_uh_T(u^{(i)})\beta  \right)&&\\
& \quad \quad + \sum_{t=0}^{T-1} \hat{L}_t\left(h_t(u^{(i)}) + D_uh_t(u^{(i)})\beta, u_t^{(i)}+\beta_t\right) + \lambda \|\beta\|^2_F \Biggr\}&&
\end{flalign*}\\
$u^{(i+1)} \leftarrow u^{(i)} +  \beta^{(i)}$\\
$i \leftarrow i+1$
}
\textbf{Return} $\left(h_1\left(u^{(i)}\right), \dots, h_T\left(u^{(i)}\right)\right)$
\vspace{2mm}
 \caption{Iterative Regression}
 \label{alg2}
\end{algorithm}

Algorithm \ref{alg2} consists in learning the optimal controls by iteratively finding the optimal controls for linear approximations of the functions $(h_t)_t$, $t = 1, \dots, T$.

Having chosen an initial set of controls $u^{(0)}$, the first-order Taylor approximation of $h_t$ around $u^{(0)}$ is
\begin{equation} \label{taylor}
h_t\left(u^{(0)}+\beta\right) \approx h_t\left(u^{(0)}\right) +  D_uh_t\left(u^{(0)}\right)\beta,
\end{equation}
having denoted as $D_uh_t\left(u^{(0)}\right)$ the gradient of $h_t$ with respect to controls $u$, the linear operator mapping $\mathds{R}^{q \times T}$ to $\mathcal{H}$.  

It is possible to learn the set of $\beta$ by solving the regularised minimisation problem
\[
\min_{\beta \in \mathds{R}^{q \times T}} \left\{F\left(h_T\left(u^{(0)}\right) +  D_uh_T\left(u^{(0)}\right)\beta  \right)  +\sum_{t=0}^{T-1} L_t\left (h_t\left(u^{(0)}\right) + D_uh_t\left(u^{(0)}\right)\beta, u_t^{(0)}+\beta_t \right) + \lambda \|\beta\|^2_F \right\},
\]
for some terminal cost function $F$ and time-dependent cost function $L$ which depend on the problem. 

The goal is to iterate this procedure in order to look for a $\beta$ which makes the linear approximations closer and closer to minimising the error functional. More precisely, we build a sequence of controls $u^{(i)}$ by minimising a cost function, as follows. At iteration $i$, for a given $u^{(i-1)}$, we look for the optimal $\beta$ providing the linear approximation \eqref{taylor} of $h_t$ around $u^{(i-1)}$ which minimises the cost functional. We then obtain $u^{(i)} = u^{(i-1)} + \beta$ and iterate with a new approximation of $h_t$ around $u^{(i)}$.

Regularisation renders itself necessary in order to prevent the controls from becoming too large, so that, when the functions $h$ are evaluated with the updated control $u^{(i-1)} + \beta$ at iteration $i$, the function does not explode. 

In order to learn the optimal controls by means of the least-square approach, we need to derive an expression  for the gradient of the function $h$ with respect to the controls $u$, which enters equation \eqref{taylor}. To this end, we exploit our knowledge of the expression of the gradient of a cost function, which was derived in Theorem \ref{lemmagrad}. Let $x$ be a query point in the state space $\mathcal{X}$. We write $h_T(x)$ as a cost function:
\[
\mathcal{E}(u):= \langle k_x,h_T(u) \rangle_{\mathcal{H}} = F(h_T(u))
\]
for $F(h) = \langle k_x,h \rangle_{\mathcal{H}}$, and proceed similarly for $h_t(x)$:
\[
\mathcal{E}(u):= \langle k_x,h_t(u) \rangle_{\mathcal{H}} = L(h_t(u),u_t),
\]
for $L(h,u)=  \langle k_x,h \rangle_{\mathcal{H}}$.

As shown in Theorem \ref{lemmagrad}, in order to find the gradient of $\mathcal{E}$ we need to solve the adjoint equation with terminal condition $h_T^*  = DF(h_T)$. When the expression of the cost function is $F(h) = \langle k_x,h \rangle_{\mathcal{H}}$, namely when we are computing the derivative of $h_T$, the derivative of $F$ is
\[
DF(h) = k_x,
\]
so that one has to solve \eqref{der1} with terminal condition $ h_T^* = k_x$. Instead, when we consider the derivative of $h_t$ for $t<T$, the terminal cost function is equal to zero, so that one has to solve Equation \eqref{der1} with terminal condition $ h_T^* = 0$. Once the solutions to both the forward and the backward equations have been computed, the desired expression for the derivative follows from Equation \eqref{der2}.

In applications, we make use of the workable specification on $L^2_{\mathds{M}}$ and we refer to Section \ref{ws} and in particular to Equation \eqref{system_bigPsi} for the discussion concerning the solution of the backward system for multiple terminal conditions. Again, the forward and backward equations for $g$ and $g^*$ have to be solved in order to evaluate the expression for the gradient. As far as $g^*$ is concerned, Lemma \ref{lemmagstar} provides its expression in terms of basis functions $\Psi$, which reduces the computational overhead of having to solve the system for several terminal conditions.


Considering for example the derivative of $h_T$, and the operator $\Psi$ solving the system of difference equations \eqref{system_bigPsi}, the terminal condition for $g^*$ can be written as
\[
g_T^* = J_{\mathds{M}}k_x,
\]
so that Equation \eqref{gfrompsi} implies that
\[
g_t^* = \Psi_t g_T^*. 
\]
Once the solutions to the forward and backward system have been computed, the expression of the gradient of $h_T$ with respect to $u$ becomes available, using Formula \eqref{grg}. Having picked an initial control $u_0$ and evaluated the gradient of $h_T$ at $u_0$, it is as a result possible to linearly approximate $h_T$ using equation \eqref{taylor}.
Similar steps lead to the linear approximation of $h_t$, $t = 1, \dots, T-1$. Once the linear approximations of the functions $h$ are available, one proceeds by iterating them in order for the linear approximation to be closer and closer to the minimisation of the error functional. The detailed steps of the method are presented in Algorithm \ref{alg2}.

Lastly, we notice that the expression \eqref{taylor} is reminiscent of the notion of tangent kernel, developed by \cite{tangkernel}. Indeed, the authors show that the neural network mapping follows the kernel gradient of the functional cost with respect to the neural tangent kernel. This establishes another link between this work, kernel methods and artificial neural networks.


\subsection{Enhanced Iterative Regression} \label{ilssec}
The third optimisation algorithm which we propose is a slight modification of Algorithm \ref{alg2}. Like Algorithm \ref{alg2}, Algorithm \ref{alg3} features an iterative procedure which aims at bringing the controls close to the optimal ones by solving an optimisation problem which involves the linearised version of $h(x)$. In addition to that, however, Algorithm \ref{alg3} includes a non-regularised regression step at the end. The output of the algorithm is the linear approximation of $h_t(x)$, denoted by $\widetilde{h_t}(x)$ and given by
\begin{equation} \label{htil}
\widetilde{h_t}=  h_t\left(u\right)+  D_uh_t\left(u^{}\right)\beta, \quad t = 1, \dots, T.
\end{equation}

\IncMargin{0mm}
\begin{algorithm}
\SetAlgoLined
\vspace{2mm}
\hspace{-5mm} \textbf{Output:} approximating functions $\left(\widetilde{h_1}\left(u^{(i)}\right), \dots, \widetilde{h_T}\left(u^{(i)}\right)\right)$.\\
 \SetKwBlock{Begin}{Iterative Regression}{end}
 \Begin{
 Apply Algorithm \ref{alg2}
 }
\SetKwBlock{Begin}{Final Step}{end}
 \Begin{
 Sample points of the mini batch $x_1^{(i)}, \dots,  x_n^{(i)}$\\ 
 Solve the following problem on the mini batch:
\begin{flalign*}
\beta^{(i)} & \leftarrow \mbox{argmin}_{\beta \in \mathds{R}^{q \times T}}  \Biggl\{\hat{F}\left(h_T(u^{(i)}) + D_uh_T(u^{(i)})\beta  \right)&& \\
& \quad \quad + \sum_{t=0}^{T-1} \hat{L}_t\left(h_t(u^{(i)}) + D_uh_t(u^{(i)})\beta, u_t^{(i)}+\beta_t\right) \Biggr\}&&
\end{flalign*}
}
\textbf{Return} $\left(\widetilde{h_1}\left(u^{(i)}\right), \dots, \widetilde{h_T}\left(u^{(i)}\right)\right)$
\vspace{2mm}
 \caption{Enhanced Iterative Regression}
 \label{alg3}
\end{algorithm}


Algorithm \ref{alg3} describes the steps that are performed, which are the same as those in Algorithm \ref{alg2} with the addition of a regression step after the iterative procedure.



\section{Examples}
This section is devoted to the description of the experimental setup and the experiments themselves. The method hitherto presented is flexible enough to lend itself to a wide set of applications. At the same time, such a flexibility comes with a set of parameters which need to be tuned, as well as precise choices for the control system.
Here, also inspired by the results of \cite{Cucchiero20}, we take $q = 2$.

Although it is in practice true that certain starting points make the optimal control easier to be found, in order to preserve the generality of the method we set the initial function $h_0$ as a constant function when performing our experiments. 

In addition, all the optimisation methods proposed require to choose an initial control $u_0$, which is then modified by a sequence of gradient steps (according to the first method -- Section \ref{gradient_descent}) or around which a linear expansion of the approximating function is built (as needed for Algorithms \ref{alg2} and \ref{alg3} - Sections \ref{least_squares} and \ref{ilssec}, respectively). Again, with a view for our set-up to be as generic as possible, we take, unless otherwise stated, our initial control to be a random draw from mean-zero gaussian random variables.
The control system also relies on the definition of the operators $B_1$ and $B_2$. We take the first one to be a diagonal operator, and the second one to be rank-one. In both cases, we construct them so that they have unitary norm, and Examples \ref{exb1} and \ref{exb2} have been devoted to the study of the properties of such operators.

A kernel is needed in order to define the reproducing-kernel Hilbert space. In our experiments, results are provided under the assumption of a Gaussian kernel
\[
K(x,x') = \exp{-\frac{\| x-x'\|^2}{2s^2}}.
\]
The scale parameter $s$ is reported for each experiment.

Further elements to be determined include the choice of a learning rate (when needed) and of the batch size $n$, the size $m$ of the inner sample $(\xi_i)_{i=1}^m$, and the time grid, namely a set of $T$ time points. 
The following subsections report the performances of the Algorithms for four different case studies and compared to several benchmark methods. Performance metrics are computed out of sample, on a test set whose size is reported for each example. The cost function in the cases considered is a terminal cost function, which changes according to the applications. This means that in the following $L(h,u) = 0$.

\subsection{Toy Example} \label{sine_text}
We first test our method on a one-dimensional example. The target function is chosen as the sine function on the interval $[-\pi,\pi]$ and the three algorithms considered are used to learn the function from a training set.

Table \ref{Tab:val_sine} reports the hyper-parameters selected for the sine problem, while Table \ref{Tab:sin} lists the values of the error metrics on a test set for each of the Algorithms considered. In addition, the table reports the same metrics for two benchmark methods. The first, a naive one, evaluates the error metric for the approximating function resulting from the initial control $u^{(0)}$, namely from the random draw used at the beginning of the optimisation algorithms. After drawing the initial control $u^{(0)}$, the naive benchmark leads to the cost $F(h(u^{(0)}))$. The second benchmark method is a standard kernel regression, namely a linear regression of the target values on the basis functions $K(\xi_i,\cdot)$, $i = 1, \dots, m$. This is a natural consequence of the fact that, as shown by \eqref{base}, the approximating function is a linear combination of the kernel functions. As a result, the accuracy of the linear regression on said kernel functions represents the natural benchmark.

\begin{table}[!h]
\begin{center}
\begin{tabular}{c c c c c c c}
  & $T$ & $m$ & $s$ & $\mu$ & $\sigma$ & $n$ \\ 
 \hline
Values & 20& 10& $10^{0.7}$ & 0& 1& 300\\  
\end{tabular}
\caption{\label{Tab:val_sine}Hyper-parameters used in the sine problem. The parameter $T$ refers to the number of time steps in the control system, $m$ to the number of points drawn from $\mathds{M}$ -- namely the size of the inner sample $(\xi_i)_{i=1}^m$, $s$ is the Gaussian kernel scale hyper-parameter, $\mu$ and $\sigma$ are the mean and standard deviation of the normal distributions from which the initial controls are sampled, $n$ is the size of the mini batch.}
\end{center}
\end{table}

\begin{table}[!h]
\begin{center}
\begin{tabular}{c c c}
  & RMSE & Iterations \\ 
 \hline
Algorithm 1 & 5.69*1e-2 & 1e6\\  
Algorithm 2 & 4.04*1e-3 & 100\\
Algorithm 3 & 2.18*1e-4 & 101\\
Naive & 0.86 & - \\ 
Kernel Regression & 8.76*1e-5 & -
\end{tabular}
\caption{\label{Tab:sin}Root-mean-square error and number of iterations for the sine example and the three algorithms considered, as well as the error for two proposed benchmarks computed on a test set of 1000 points. The training data for the Kernel Regression (KR) was taken of size equal to 10000.}
\end{center}
\end{table}
The one-dimensional example allows to visually assess the performance of the algorithms.
Figures \ref{fig:sine1} and \ref{fig:sine2} show the performance of the method when the controls are learnt by Algorithms \ref{alg2} and \ref{alg3} and the target function is the sine function over the interval $[-\pi,\pi]$. Two cases are illustrated in Figure \ref{fig:sine1}. The label \textit{Linear Approximation} refers to the approach which yields the linear approximation \eqref{htil} after a number $\tau$ of iterations. In the second case, referred to as \textit{Control System}, the approximation function is the solution to the control system with optimal controls $u_{\tau} + \beta$, namely $h_T(u_{\tau} + \beta)$. This corresponds to the controls obtained at iteration $\tau$ with the addition of a step equal to $\beta$.

\begin{figure}
        \centering
            \includegraphics[width=0.475\textwidth]{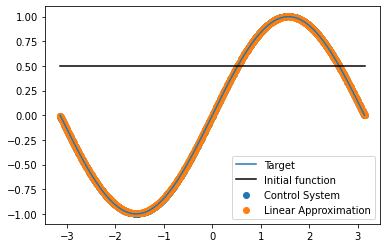}
            \caption{Sine function learnt via Algorithm 2.}    
            \label{fig:sine1}
\end{figure}
\begin{figure}
        \centering
            \includegraphics[width=0.475\textwidth]{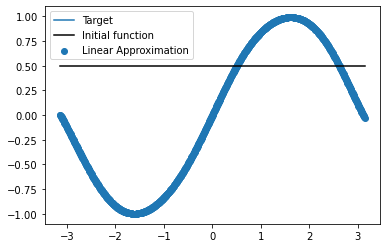}
            \caption{Sine function learnt via Algorithm 3.}    
            \label{fig:sine2}
\end{figure}

\subsection{Multi-dimensional Toy Example}
The algorithms are further tested on a three-dimensional example. Specifically, the target function is chosen as
\[
f(x_1,x_2,x_3) = 0.5,x_1 -0.2 x_2 +0.1x_3,
\]
on the hyper cube $[-3,3]^{3}$.

We use the three algorithms considered to learn a set of optimal controls and test them on a test set of 10000 points. In addition, we compare our method to a kernel regression and to the naive approach, as described in Section \ref{sine_text}. Table \ref{Tab:val_toy3} reports the hyper-parameters selected for the multi-dimensional linear problem. Table \ref{Tab:3d} reports the root mean-square error for the three approaches considered and for the benchmark methods. 

\begin{table}[!h]
\begin{center}
\begin{tabular}{c c c c c c c}
  & $T$ & $m$ & $s$ & $\mu$ & $\sigma$ & $n$ \\ 
 \hline
Values & 50& 10& 1& 0& 0.1& 1000\\  
\end{tabular}
\caption{\label{Tab:val_toy3}Hyper-parameters used in the three-dimensional problem. The parameter $T$ refers to the number of time steps in the control system, $m$ to the number of points drawn from $\mathds{M}$, $s$ is the Gaussian kernel hyper-parameter, $\mu$ and $\sigma$ are the mean and standard deviation of the normal distributions from which the initial controls are sampled, $n$ is the size of the mini batch.}
\end{center}
\end{table}

\begin{table}[!h]
\begin{center}
\begin{tabular}{c c c}
  & RMSE & Iterations\\ 
 \hline
Algorithm 1 &1.43*1e-1 & 1e6\\  
Algorithm 2 & 3.43*1e-4 & 5*1e5\\
Algorithm 3 & 3.56*1e-4 & 5*1e5+1\\
Naive & 1.5\\
Kernel Regression & 1.11*1e-4
\end{tabular}
\caption{\label{Tab:3d}Root-mean-square error and number of iterations for the three-dimensional toy example and the three algorithms considered, as well as the error for two proposed benchmarks. The training data for the Kernel Regression (KR) was taken of size equal to 10000.}
\end{center}
\end{table}
Although the kernel regression performs slightly better, Algorithms \ref{alg2} and \ref{alg3} give results which are in the same order of magnitude, namely $10^{-4}$. All these methods significantly outperform a gradient-descent approach. Figure \ref{fig:3d_a} shows the distribution of the error of the approximation resulting from the solution of the control system (Algorithm \ref{alg2}). The plot in the case of the linear approximation (Algorithm \ref{alg3}) is similar.
\begin{figure}
        \centering
            \includegraphics[width=0.475\textwidth]{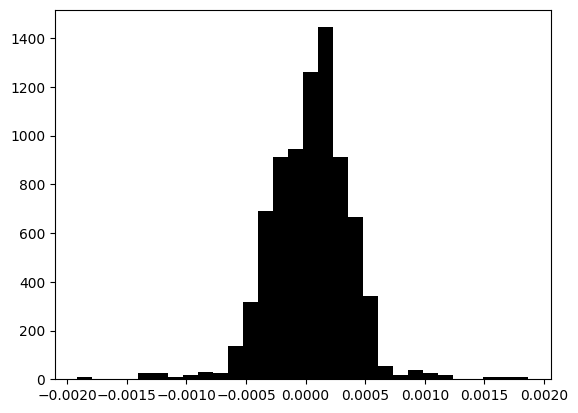}
            \caption{Distribution of the prediction error on a test set of 10000 points for iterated regression on the three-dimensional toy example.}    
            \label{fig:3d_a}
\end{figure}

\subsection{Regression}
We consider the task of learning a pricing function. The goal is to be able to price European call options under the Heston model for a set of given parameters. In order to do this, one first exploits a stochastic model to compute a set of prices for several choices of the parameters of interest. This constitutes the training set, to which the method proposed in the previous sections is applied with the aim of learning the pricing function. The underlying principle is that a function provides prices faster than what would be the case should one solve the stochastic model each time for a new choice of the parameters. 

We consider a European call option, namely a contract with payoff 
\[
f(S_T) = (S_T-K)_{+},
\]
where $S_t$ denotes the price of the underlying stock at time $t$ and $K$ and $T$ are respectively the strike price and the maturity, characteristics of the option. The underlying stock price is described by the Heston model, whose dynamics for $S$ is given by
\[
dS_t = S_t \left(r \, dt + \sqrt{V_t} \, dW^{(1)}_t\right),
\]
for an interest rate $r$ and a Brownian motion $W^{(1)}_t$. The variance process $V_t$ is also stochastic and follows the dynamics
\[
dV_t = \kappa(\theta-V_t) \, dt + \sigma \sqrt{V_t} dW^{(2)}_t.
\]
The two Brownian motions $W^{(1)}$ and $W^{(2)}$ have correlation $\rho$. In addition, $\kappa$ is the mean-reversion speed, $\theta$ is the mean-reversion level of the variance process.

We use the Fast Fourier Transform (FFT) approach based on the characteristic function of the Heston model to compute prices for a set of parameters. Specifically, for several combinations of parameter values in a given range we compute prices and form a training set. We choose parameters to take values as reported in the following Table \ref{Tab:Tcr}.
\begin{table}[!h]
\begin{center}
\begin{tabular}{c c}
 Parameter & Values \\ 
 \hline
$K$ & [50,150] \\  
$T$ & [11/12,1]\\
$r$ & [0.015,0.025] \\
$\kappa$ & [1.5,2.5]\\
$\theta$ & [0.5,0.7]\\
$\rho$ & [-0.7,-0.5]\\
$\sigma$ & [0.02,0.1]\\
$V_0$ & [0.02,0.1]
\end{tabular}
\caption{\label{Tab:Tcr}Ranges of values for the selected model and market parameters.}
\end{center}
\end{table}

\begin{table}[!h]
\begin{center}
\begin{tabular}{c c c c c c c}
  & $T$ & $m$ & $s$ & $\mu$ & $\sigma$ & $n$ \\ 
 \hline
Values & 20& 500& $10^{0.1}$ &0& 20& 1000 \\  
\end{tabular}
\caption{\label{Tab:val_eu} Hyper-parameters used in the regression problem. The parameter $T$ refers to the number of time steps in the control system, $m$ to the number of points drawn from $\mathds{M}$, $s$ is the Gaussian kernel hyper-parameter, $\mu$ and $\sigma$ are the mean and standard deviation of the normal distributions from which the initial controls are sampled, $n$ is the size of the mini batch.}
\end{center}
\end{table}

Table \ref{Tab:val_eu} reports the parameters used in the optimisation process. The accuracy of the methods proposed in the case of the pricing problem is shown in Table \ref{Tab:eu}. As done previously, we consider all the optimisation methods described in Section \ref{opti}. The errors are computed relative to a benchmark, the Fast Fourier Transform.

\begin{table}[!h]
\begin{center}
\begin{tabular}{c c c c}
  & RMSE & MAPE & Iterations\\ 
 \hline
Algorithm 1 & 1.12*1e-1&5.2*1e-1 & 3*1e4\\  
Algorithm 2 & 1.89*1e-2& 8.61*1e-2 &1000\\
Algorithm 3 & 1.12*1e-2& 4.9*1e-2 & 1001\\
Naive & 7.50*1e-1& 3.95 & -\\
GPR & 3.02*1e-3& 1.50*1e-3 & -\\
Kernel Regression & 1.01*1e-3 & 2.20*1e-3
\end{tabular}
\caption{\label{Tab:eu}Root-mean-square error and mean-absolute-percentage error for the pricing problem and the algorithms considered, as well as for two benchmark methods, a gaussian-process regression and a naive approach. The table also shows the number of iterations used for the algorithms considered. The training data for the Gaussian-Process Regression (GPR) and for the Kernel Regression was taken of size equal to 10000.}
\end{center}
\end{table}
Figure \ref{fig:eu_a} shows the distribution of the error of the approximation resulting from the linear approximation (Algorithm \ref{alg3}). The plot in the case of Algorithm \ref{alg2} is analogous, unlike that for Algorithm \ref{alg1}, which is shown to underperform relative to the other approaches.

\begin{figure}
        \centering
            \includegraphics[width=0.475\textwidth]{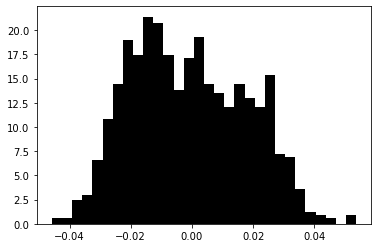}
            \caption{Distribution of the prediction error on a test set of 1000 points for iterated regression.}    
            \label{fig:eu_a}
\end{figure}

\subsection{Classification}
We consider the Higgs boson detection problem from Kaggle. This competition (\url{https://www.kaggle.com/c/higgs-boson/overview}) aims at developing a model which can identify observations of the Higgs boson based on observed physical quantities. The task involves the classification of events into tau tau decay of a Higgs boson versus background. The problem of the Higgs boson identification is currently one of the most discussed in the field of particle physics, and machine learning offers promising approaches which can be deployed to come up with novel solutions. In this dataset, observations are provided by CERN from the ATLAS experiment. The training data (\texttt{training.csv}) consists of 250000 events, with an ID column, 30 feature columns, a weight column and a label column. The classification problem is that of assigning events to a label ($\{0,1\}$) based on the values of the features. 65\% of the data take label 0.

We consider two alternative approaches to the learning problem in a classification setting, with two different error functionals: the cross entropy function and a quadratic cost.

In the first case, the error functional is the expectation \eqref{crosse} of the cross entropy \eqref{crosse0}. An explicit expression for the derivative of the terminal cost function was derived in Example \ref{logregex} and allows to learn the set of optimal controls by gradient descent, as described in Section \ref{gradient_descent}. The other optimisation methods proposed are also tested, using the suitable cost function in Algorithms \ref{alg2} and \ref{alg3}.

The alternative approach, following \cite{James2013}, consists of working with the quadratic terminal costs \eqref{cost1}, thus leading to a least-square problem.
Given the binary nature of the response in the problem that we consider, it can be shown that least squares are equivalent to linear discriminant analysis (LDA). In turn, there exist cases where using a cross-entropy error functional leads to unstable estimates, and other approaches, such as LDA, are to be preferred (see again \cite{James2013}). The quadratic terminal cost \eqref{cost1} lends itself to the optimisation methods described in Sections \ref{gradient_descent}, \ref{least_squares} and \ref{ilssec}.

We implement the approaches hitherto described and find the one based on regression with a quadratic terminal cost (Algorithms \ref{alg2} and \ref{alg3}) to perform best, and only slightly worse compared to traditional methods, such as ridge regression with feature expansion. Tables \ref{Tab:hig21} and \ref{Tab:hig22} show two relevant metrics in a classification context, the accuracy and the F1 Score, for the two terminal costs, the cross entropy and the quadratic cost. Accuracy values for other standard methods are reported in Table \ref{Tab:hig3}, showing an adequate performance of our proposed method.

\begin{table}[!h]
\begin{center}
\begin{tabular}{c c c c c c c}
  & $T$ & $m$ & $s$ & $\mu$ & $\sigma$ & $n$ \\ 
 \hline
Values & 12& 100& $10^{1.5}$& 0& 10& 1000\\  
\end{tabular}
\caption{\label{Tab:val_class}Hyper-parameters used in the classification problem. The parameter $T$ refers to the number of time steps in the control system, $m$ to the number of points drawn from $\mathds{M}$, $s$ is the Gaussian kernel hyper-parameter, $\mu$ and $\sigma$ are the mean and standard deviation of the normal distributions from which the initial controls are sampled, $n$ is the size of the mini batch.}
\end{center}
\end{table}

\begin{table}[!h]
\begin{center}
\begin{tabular}{c c c}
  & Accuracy & F1 Score\\
\hline
Algorithm 1 & 0.68 & 0.26 \\
Algorithm 2 & 0.77 & 0.61\\
Algorithm 3 &  0.81& 0.71\\
\end{tabular}
\caption{\label{Tab:hig21} Metrics on a test set of 1000 observations for the proposed algorithms and the $L^2$ cost.}
\end{center}
\end{table}

\begin{table}[!h]
\begin{center}
\begin{tabular}{c c c}
  & Accuracy & F1 Score\\
\hline
Algorithm 1 &  0.69 & 0.20\\
Algorithm 2 & 0.68 & 0.61\\
Algorithm 3 &  0.70& 0.38\\
\end{tabular}
\caption{\label{Tab:hig22} Metrics on a test set of 1000 observations for the proposed algorithms and the cross-entropy cost.}
\end{center}
\end{table}

\begin{table}[!h]
\begin{center}
\begin{tabular}{c c c}
Method & Accuracy & F1 Score\\
\hline
Ridge & 0.83 & 0.87\\
Lasso & 0.72 & 0.86\\
Logit & 0.69 & 0.65\\
Kernel Regression & 0.89 & 0.81
\end{tabular}
\caption{\label{Tab:hig3} Metrics on a test set of 1000 observations for alternative methods. The training data was taken of size equal to 10000 for the last approach, while the full dataset was used for the other ones.}
\end{center}
\end{table}

Figure \ref{fig:higfig} shows the histogram of the values obtained by the approximating function in the $L^2$ case and for Algorithm \ref{alg2}, which are for the most part in the range $[0,1]$, as desired.
\begin{figure}
        \centering
            \includegraphics[width=0.475\textwidth]{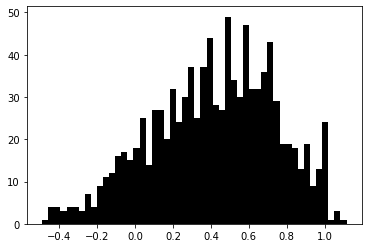}
            \caption{Distribution of the values obtained by the approximating function.}    
            \label{fig:higfig}
\end{figure}

\section{Conclusion} Our approach shows how to learn from data by following a control-theoretical perspective. Having parametrised the approximating function as the solution to a control system on a Hilbert space, we propose several algorithms to find the set of controls which bring the initial function as close as possible to the target function. We derive the expression for the gradient of the cost function with respect to the controls that parametrise the difference equations and propose a gradient-descent based method to find the optimal controls. In addition, we show how to compute derivatives of the approximating functions with respect to said controls and propose two alternative optimisation methods relying on linear approximations of the approximating functions around a fixed control. 

Interestingly, we show how the assumptions we make lead to results which are coherent with Pontryagin's maximum principle.  We test the proposed optimisation algorithms on two toy examples and on two higher-dimensional real world problems, showing that our method succeeds in learning from real data and is versatile enough to tackle learning tasks of different nature.

As well as introducing a novel learning approach, our paper fully links control theory and learning theory. Control theory is explicitly being used within the learning task and control-theoretical results are naturally recovered as part of the learning process. Our paper also extends the previous literature, which is typically based on a state space in $\mathds{R}^n$, to Hilbert spaces, and to bilinear systems, rather than linear ones. In addition, our approach allows to view a control problem in the light of a learning task, which is especially relevant in the current research environment.

Further research shall focus on alternative parametrisations of the approximating functions and of different specifications of the control system in order to enhance the accuracy of the approximating function. 


\bibliographystyle{chicago}
  \bibliography{references2_.bib}

\end{document}